\documentclass[12pt, a4paper, english]{amsart}
\usepackage{amsmath} 
\usepackage{amsthm} 
\usepackage{amssymb} 
\usepackage{amscd} 
\usepackage{array} 
\usepackage[mathscr]{eucal} 
\usepackage[backref=page]{hyperref} 
\usepackage{caption} %
\usepackage{graphics,graphicx} 
\usepackage{tikz,tikz-cd} 
\usepackage{enumerate} 
\usepackage{lipsum} 
\usepackage[margin=1in]{geometry}
\usetikzlibrary{graphs,graphs.standard,calc}

\hypersetup{
    colorlinks = true,
    linkbordercolor = {black},
    linkcolor = {purple},
    anchorcolor = {black},
    citecolor = {blue},
    filecolor = {cyan},
    menucolor = {red},
    runcolor = {cyan},
    urlcolor = {magenta}
}

\usetikzlibrary{automata}

\newtheoremstyle{teoremas}
{14pt}
{14pt}
{\itshape}
{}
{\bfseries}
{}
{.5em}
{}
\theoremstyle{teoremas}
\newtheorem{teo}{Theorem}[section]
\newtheorem{cor}[teo]{Corollary}

\newtheorem{lema}[teo]{Lemma}
\newtheorem{question}[teo]{Question}
\newtheorem{prop}[teo]{Proposition}

\newtheoremstyle{definition}
{12pt}
{12pt}
{}
{}
{\bfseries}
{}
{.5em}
{}
\theoremstyle{definition}

\newtheorem{defi}[teo]{Definition}

\newtheorem{conj}[teo]{Conjecture}
\newtheorem{ej}[teo]{Example}
\newtheorem{ex}[teo]{Example}
\newtheorem{oss}[teo]{Remark}
\newtheorem{obs}[teo]{Remark}

\newcommand{\M}{\mathsf{M}}
\newcommand{\N}{\mathsf{N}}
\newcommand{\U}{\mathsf{U}}
\newcommand{\T}{\mathsf{T}}
\newcommand{\B}{\mathsf{B}}
\newcommand{\W}{\mathsf{W}}
\DeclareMathOperator{\rk}{rk}

\title{Matroid relaxations and Kazhdan--Lusztig non-degeneracy}

\author[L. Ferroni \& L. Vecchi]{Luis Ferroni and Lorenzo Vecchi}
\thanks{The first author was supported by the Marie Sk{\l}odowska-Curie PhD fellowship as part of the program INdAM-DP-COFUND-2015, Grant Number 713485 and also partially supported by Grant 2018-03968 of the Swedish research council.}

\address{Department of Mathematics, KTH Royal Institute of Technology, Stockholm, Sweden} 
\email{ferroni@kth.se}
\address{Universit\`a di Bologna, Dipartimento di Matematica, Piazza di Porta San Donato, 5, 40126 Bologna, Italy} 
\email{lorenzo.vecchi6@unibo.it}


\allowdisplaybreaks

\begin{document}

\begin{abstract}
    In this paper we study the interplay between the operation of circuit-hyperplane relaxation and the Kazhdan--Lusztig theory of matroids. We obtain a family of polynomials, not depending on the matroids but only on their ranks, that relate the Kazhdan--Lusztig, the inverse Kazhdan--Lusztig and the $Z$-polynomial of each matroid with those of its relaxations. As an application of our main theorem, we prove that all matroids having a free basis are non-degenerate. Additionally, we obtain bounds and explicit formulas for all the coefficients of the Kazhdan--Lusztig, inverse Kazhdan--Lusztig and $Z$-polynomial of all sparse paving matroids.\\
    
\end{abstract}

\maketitle

\section{Introduction}

\subsection{Overview} Given a Coxeter group $W$, to each pair of elements that are comparable with respect to the Bruhat order $x\leq y$ it is possible to associate a polynomial $P_{x,y}(t)$ having integer coefficients. Such polynomials are known in the literature as the Kazhdan--Lusztig polynomials of $W$ and encode fundamental information about the combinatorial and algebro-geometric features of the Coxeter group.

In 2016 Elias et al. \cite{eliasproudfoot} introduced an analog of the Kazhdan--Lusztig polynomial for matroids. The role played by the Bruhat poset in the Coxeter setting is now played by the lattice of flats of the matroid, although now instead of defining a particular polynomial for each pair of comparable flats, we can define a unique polynomial for the whole matroid.

These two theories can be seen as particular cases within the broader framework of Kazhdan--Lusztig--Stanley polynomials \cite{stanley,brenti}. They share certain similarities: for example both in the Coxeter setting \cite{eliaswilliamson} and in the matroid setting \cite{bradenhuh} the resulting polynomials have non-negative coefficients. There are, however, manifest differences between the type of polynomials that can arise as the Kazhdan--Lusztig polynomial of a matroid as opposed to a Kazhdan--Lusztig polynomial of a pair of elements of a Coxeter group. For example, Gedeon et al. posed the following conjecture.

\begin{conj}[{\cite[Conjecture 3.2]{surveygedeon}}]\label{realroots}
    The Kazhdan--Lusztig polynomial of a matroid is real-rooted.
\end{conj}

This conjecture, if true, would establish a profound difference between these two settings. This is because it is known that every polynomial with non-negative coefficients and constant term $1$ arises as the Kazhdan--Lusztig polynomial of a pair of elements in a suitable Coxeter group \cite{polo}. 

To approach the above conjecture, in \cite{surveygedeon} Gedeon et al. proposed a line of attack by conjecturing that there are interlacing properties for the roots of the Kazhdan--Lusztig polynomial of a matroid $\M$ and those of a contraction $\M/\{e\}$ by one element. As interlacing of roots makes sense only if the degrees of the considered polynomials differ at most by one, it is natural to ask about the degree of the Kazhdan--Lusztig polynomials of matroids.

A matroid $\M$ is said to be \emph{non-degenerate} if its Kazhdan--Lusztig polynomial has the maximal possible degree, namely $\lfloor\frac{\rk(\M) - 1}{2}\rfloor$.

\begin{conj}[{\cite[Conjecture 2.5]{surveygedeon}}]\label{conject}
    Every connected regular matroid is non-degenerate.
\end{conj}

On the other hand, since its introduction in \cite{eliasproudfoot}, the Kazhdan--Lusztig theory of matroids has been approached under different perspectives. In particular, some families of polynomials related to the Kazhdan--Lusztig polynomials have been introduced for matroids. In \cite{proudfootzeta} Proudfoot et al. introduced the so-called ``$Z$-polynomial'' of a matroid. Also, in \cite{gaoinverse} Gao and Xie defined the ``inverse Kazhdan--Lusztig polynomial'' of a matroid. It is now customary to use the notation $P_{\M}$ for the Kazhdan--Lusztig polynomial, $Q_{\M}$ for the inverse Kazhdan--Lusztig polynomial and $Z_{\M}$ for the $Z$-polynomial. 

The definitions of these three families of polynomials, namely $P_\M$, $Q_\M$ and $Z_\M$, are recursive. It is computationally expensive to calculate them explicitly using their defining recurrence. For some particular families of matroids such as uniform matroids \cite{gaouniform,gaoinverse}, wheels and whirls \cite{wheels} and other graphic matroids \cite{thagomizer}, it is possible to derive concrete and explicit formulas. Although these particular cases are interesting, the number of matroids belonging to each of these classes is almost negligible within the family of all matroids. On the other hand, in \cite{lee} a combinatorial formula for $P_\M$ when $\M$ is a sparse paving matroid is obtained via enumerating certain (skew) Young tableaux. The family of sparse paving matroids is conjectured to predominate asymptotically \cite{mayhew}. 

\subsection{Outline and main results}

After reviewing in Section \ref{sec:two} the basics of matroid theory and Kazhdan--Lusztig polynomials, in Section \ref{sec:three} we study the interplay between the Kazhdan--Lusztig theory of a matroid and that of its ``relaxations.'' This classical operation allows to enlarge the set of bases of a matroid by one extra element. We prove the following key fact from which we derive the remaining main results of the present article.

\begin{teo}\label{mainresult2}
    For each integer $k\geq 1$ there exist polynomials $p_k$, $q_k$ and $z_k$ with integer coefficients such that for every matroid $\M$ of rank $k$ having a circuit-hyperplane $H$ the following equalities hold:
        \begin{align*} 
            P_{\widetilde{\M}}(t) &= P_{\M}(t) + p_k(t),\\
            Q_{\widetilde{\M}}(t) &= Q_{\M}(t) + q_k(t),\\
            Z_{\widetilde{\M}}(t) &= Z_{\M}(t) + z_k(t),
        \end{align*}
    where $\widetilde{\M}$ denotes the corresponding relaxation of $\M$ by $H$. Moreover, $p_k$, $q_k$ and $z_k$ have non-negative coefficients, $\deg p_k = \deg q_k = \lfloor \frac{k-1}{2}\rfloor$ and $\deg z_k = k-1$.
\end{teo}

As a consequence of the known formulas for $P_{\M}$ and $Q_{\M}$ when $\M$ is a uniform matroid, and $Z_{\M}$ when $\M$ is a wheel matroid, in Corollary \ref{formulitas} we are able to deduce closed formulas for each of $p_k$, $q_k$ and $z_k$. Moreover, we discuss how Theorem \ref{mainresult2} in some sense exhibits explicitly a particular case of the fact that the Kazhdan--Lusztig polynomial and inverse Kazhdan--Lusztig polynomial are valuative invariants of the matroid polytope \cite[Theorems 8.8 and 8.9]{ardila} (i.e. they behave nicely for matroidal subdivisions of the base polytope).

Turning back our attention to the non-degeneracy of matroids, in Section \ref{sec:four} we study the notion of ``free basis'' of a matroid. We say that the basis $B$ of the matroid $\M$ is \emph{free} if $\M$ has at least two bases and $B$ has the property that for every element $e\notin B$, $B\cup\{e\}$ is a circuit. In \cite[Conjecture 22]{bansal} Bansal et al. posed a conjecture that essentially states that asymptotically almost all matroids have a free basis. From Theorem \ref{mainresult2} we deduce the following fact.

\begin{teo}\label{mainresult1}
    If a matroid has a free basis, then it is non-degenerate.
\end{teo}

In Section \ref{sec:five} we address sparse paving matroids. This class is conjectured to be predominant among all matroids \cite[Conjecture 1.6]{mayhew}. Using Theorem \ref{mainresult2} we prove the following result.

\begin{teo}\label{mainresult3}
    Let $\M$ be a sparse paving matroid of rank $k$ and cardinality $n$. Assume that $\M$ has exactly $\lambda$ circuit-hyperplanes. Then:
    \begin{align*}
        P_{\M}(t) &= P_{\U_{k,n}}(t) - \lambda p_k(t),\\
        Q_{\M}(t) &= Q_{\U_{k,n}}(t) - \lambda q_k(t),\\
        Z_{\M}(t) &= Z_{\U_{k,n}}(t) - \lambda z_k(t).
    \end{align*}
\end{teo}

In other words, we have derived closed formulas for the polynomials $P_{\M}$, $Q_\M$ and $Z_\M$ for all sparse paving matroids. This recovers and extends results by Lee et al. \cite{lee}, which were valid only for $P_{\M}$. In particular, as a consequence of the non-negativity of our polynomials $p_k$, $q_k$ and $z_k$, we are able to conclude the following statement.

\begin{teo}
    The uniform matroid $\U_{k,n}$ maximizes the Kazhdan--Lusztig coefficients, the inverse Kazhdan--Lusztig coefficients, and the $Z$-polynomial coefficients among all sparse paving matroids of rank $k$ and cardinality $n$.
\end{teo}

Also, as another application of Theorem \ref{mainresult2}, we give an elementary proof of the following fact.

\begin{teo}
    If $\M$ is a sparse paving matroid, then $Q_{\M}$ and $Z_{\M}$ have non-negative coefficients.
\end{teo}

The formulas derived from Theorem \ref{mainresult3} might be useful for testing several statements such as Conjecture \ref{realroots}. For example, combining it with the bound discussed in Lemma \ref{numbercircuithyperplanesparsepaving}, now it is possible to verify in a reasonable amount of time with a computer that:

\begin{prop}
    If $\M$ is a sparse paving matroid with at most $30$ elements, then $P_{\M}$ and $Z_{\M}$ have real roots and $Q_{\M}$ has log-concave coefficients.
\end{prop}

In \cite{ferroni3} the first author used a similar strategy to disprove a conjecture in Ehrhart theory asserting that the coefficients of the Ehrhart polynomials of matroid polytopes were always non-negative.

Finally, in Section \ref{sec:six} we discuss what are the cases of Conjecture \ref{conject} covered by our Theorem \ref{mainresult1}; we characterize all the connected, regular matroids that have a free basis. Also, we present an example of a degenerate matroid that is not modular that was communicated to us by Nicholas Proudfoot.

\section{Basics and Notations}\label{sec:two}

In this section we recall the basic definitions about matroids and their Kazhdan--Lusztig theory, and establish the notation we will use throughout the paper. For any undefined concept, we refer to Oxley's book on matroid theory \cite{oxley}.

\subsection{Matroid theory}

\begin{defi}
    A matroid $\M=(E,\mathscr{B})$ consists of a finite set $E$ and a family of subsets $\mathscr{B}\subseteq 2^E$ that satisfies the following two conditions.
    \begin{enumerate}[(a)]
        \item $\mathscr{B}\neq \varnothing$.
        \item If $B_1\neq B_2$ are members of $\mathscr{B}$ and $a\in B_1\smallsetminus B_2$, then there exists an element $b\in B_2\smallsetminus B_1$ such that $(B_1\smallsetminus \{a\})\cup \{b\}\in \mathscr{B}$.
    \end{enumerate}
\end{defi}

We usually refer to condition (b) as the \emph{basis-exchange-property}. The size of the ground set $E$ is usually referred to as the \emph{cardinality} or the \emph{size} of the matroid. An element $e\in E$ that does not belong to any basis $B$ is said to be a \emph{loop}, whereas an element that belongs to all the bases of $\M$ is said to be a \emph{coloop}. Sometimes we will focus on \emph{loopless} matroids, i.e. matroids that do not have any loops.

One of the classic examples of matroids is that of uniform matroids. Throughout this article we will denote by $\U_{k,n}$ the uniform matroid of rank $k$ and cardinality $n$. Concretely, $\U_{k,n}$ is defined by $E=\{1,\ldots, n\}$ and $\mathscr{B} = \{ B\subseteq E: |B| = k\}$. We also denote with $\B_n=\U_{n,n}$ the Boolean matroid of rank $n$. The matroid given by $E = \varnothing$ and $\mathscr{B} = \{\varnothing\}$ is referred to as the \emph{empty matroid}; notice that it is the only matroid of cardinality $0$.\\

{\bf Warning}: in some other articles, the notation $\U_{m,d}$ stands for the uniform matroid of rank $d$ and $m+d$ elements. \\

There are several basic concepts about matroids that we will use throughout this article. 

\begin{defi}
    Let $\M=(E,\mathscr{B})$ be a matroid.
    \begin{itemize}
        \item If $I\subseteq E$ is contained in some $B\in \mathscr{B}$, we say that $I$ is \emph{independent}.
        \item If $A\subseteq E$ is not independent, we say that $A$ is \emph{dependent}.
        \item If $C\subseteq E$ is dependent but every proper subset of $C$ is independent, we say that $C$ is a \emph{circuit}.
        \item For every $A\subseteq E$ we define its \emph{rank} to be:
            \[ \rk_{\M}(A) = \max_{B\in \mathscr{B}} |A\cap B|.\]
        We say that the rank of $\M$ is just $\rk_{\M}(E)$.
        \item If $F\subseteq E$ is a subset such that for every $e\notin F$ we have $\rk_{\M}(F\cup \{e\}) > \rk_{\M}(F)$, we say that $F$ is a \emph{flat}. The family of all flats of $\M$ will be denoted by $\mathscr{F}(\M)$.
        \item If $H\subseteq E$ is a flat and $\rk_{\M}(H) = \rk_\M(E) - 1$, then we say that $H$ is a \emph{hyperplane}.
    \end{itemize}
\end{defi}

Each of these notions can be used to give equivalent definitions for matroids. When $\M$ is clear from context it is customary to just write $\mathscr{F}$, $\rk$, etc. 

Motivated by certain graph-theoretic classical concepts, it is possible to define ``connectedness'' and ``simplicity'' for matroids. The matroid $\M$ is said to be \emph{connected} if for every pair of elements $x\neq y$ in the ground set it is possible to find a circuit containing both $x$ and $y$. Furthermore, $\M$ is said to be \emph{simple} if it does not have loops nor circuits of size $2$ (namely, \emph{parallel elements}).

\subsection{Kazhdan--Lusztig polynomials}

Remarkably, the family $\mathscr{F}(\M)$ of flats of a matroid $\M$, when looked as a poset with respect to inclusion, is a geometric lattice, i.e. a lattice which is both atomistic and semi-modular. It is a fundamental result of matroid theory that up to simplification of the matroid, geometric lattices and matroids are in one-to-one correspondence. The lattice of flats of a matroid $\M$ will be denoted by $\mathscr{L}(\M)$. Since for a ranked poset one can define its ``characteristic polynomial'', and geometric lattices are ranked, we can define the characteristic polynomial $\chi_{\M}$ of a matroid $\M$ to be just the characteristic polynomial of its lattice of flats.

\begin{defi}\label{def:restriction-contraction}
    Let $\M=(E,\mathscr{B})$ be a matroid and fix $F\in \mathscr{F}(\M)$. We define:
    \begin{enumerate}[(a)]
        \item $\M_F$ the matroid with ground set $E\smallsetminus F$ and with family of flats given by the sets $F'\smallsetminus F$ for all the flats $F'\in \mathscr{F}(\M)$ such that $F'\supseteq F$.
        \item $\M^F$ the matroid with ground set $F$, whose flats are the flats contained in $F$.
    \end{enumerate}
\end{defi}

The fact that these two objects are indeed matroids can be proved by noticing that their lattices of flats are isomorphic to the intervals in $\mathscr{L}(\M)$ given by $[F,\widehat{1}]$ and $[\widehat{0}, F]$ respectively, where by $\widehat{0}$ and $\widehat{1}$ we denote the bottom and top element of the lattice. In \cite[Chapter 1-3]{oxley} the operations of ``restriction'' and ``contraction'' for matroids are defined. The matroid that we denote by $\M_F$ is the matroid obtained by contracting the flat $F$ in $\M$. On the other hand, $\M^F$ is the matroid obtained by restricting $\M$ to the flat $F$. In \cite{oxley} and a significant part of the matroid theory literature these two matroids are also denoted by $\M/F$ and $\M|F$.

\begin{obs}
    The notation for $\M^F$ and $\M_F$ is usually a cause of confusion. For example, in \cite{eliasproudfoot} the authors have denoted $\M^F$ and $\M_F$ in the opposite way. The notation we used for our Definition \ref{def:restriction-contraction} follows \cite{bradenhuh}.
\end{obs}

\begin{teo}[{\cite[Theorem 2.2]{eliasproudfoot}}]\label{PM}
    There is a unique way to assign to each matroid $\M$ a polynomial $P_{\M}(t)\in \mathbb{Z}[t]$ such that the following three conditions hold:
    \begin{enumerate}
        \item If $\rk(\M) = 0$, then $P_{\M}(t) = 1$ when $\M$ is empty, and $P_{\M}(t)=0$ otherwise.
        \item If $\rk(\M) > 0$, then $\deg P_{\M} < \frac{\rk(\M)}{2}$.
        \item For every $\M$, the following holds:
            \[ t^{\rk(\M)} P_{\M}(t^{-1}) = \sum_{F\in \mathscr{L}(\M)} \chi_{\M^F}(t) P_{\M_F}(t).\]
    \end{enumerate}
\end{teo}

\begin{teo}[{\cite[Theorem 1.2]{gaoinverse}}]\label{QM}
   There is a unique way to assign to each matroid $\M$ a polynomial $Q_{\M}(t)\in \mathbb{Z}[t]$ such that the following three conditions hold:
    \begin{enumerate}
        \item If $\rk(\M) = 0$, then $Q_{\M}(t) = 1$ when $\M$ is empty, and $Q_{\M}(t)=0$ otherwise.
        \item If $\rk(\M) > 0$, then $\deg Q_{\M} < \frac{\rk(\M)}{2}$.
        \item For every $\M$, the following holds:
            \[ (-t)^{\rk(\M)} Q_{\M}(t^{-1}) = \sum_{F\in \mathscr{L}(\M)} (-1)^{\rk(F)}Q_{\M^F}(t) t^{\rk(\M) - \rk (F)}\chi_{\M_F}(t^{-1}).\]
    \end{enumerate}    
\end{teo}

The polynomials $P_{\M}$ and $Q_{\M}$ arising from the above two results are called the \emph{Kazhdan--Lusztig} and the \emph{inverse Kazhdan--Lusztig polynomials} of the matroid $\M$, respectively. Also, in \cite[Definition 2.1]{proudfootzeta} the following definition is given.

\begin{defi}
    For every matroid $\M$ we define the \emph{$Z$-polynomial} of $\M$ to be:
        \[ Z_{\M}(t) = \sum_{F\in \mathscr{L}(\M)} t^{\rk(F)} P_{\M_F}(t).\]
\end{defi}

\begin{oss}\label{looplessness}
    In \cite{eliasproudfoot} and \cite{gaoinverse} $P_\M$ and $Q_\M$ are defined only for loopless matroids. The definitions here include the cases in which the matroids contain loops. From these definitions it follows that $P_{\M}(t)=0$ and $Q_{\M}(t)=0$ whenever $\M$ contains a loop.
\end{oss}

\section{Relaxations of a Matroid}\label{sec:three}

\subsection{Circuit-hyperplanes and relaxations}

When a matroid $\M$ possesses a subset $H$ that is at the same time a circuit and a hyperplane, we say that $H$ is a \emph{circuit-hyperplane}. It is straightforward to check that every circuit-hyperplane has cardinality equal to the rank of the matroid. An important feature of circuit-hyperplanes is that they can be ``relaxed'' in order to build a new matroid with one extra basis. More precisely:

\begin{prop}
	Let $\M=(E,\mathscr{B})$ be a matroid that has a circuit-hyperplane $H$. Let $\widetilde{\mathscr{B}}=\mathscr{B}\cup\{H\}$. Then $\widetilde{\mathscr{B}}$ is the set of bases of a matroid $\widetilde{\M}$ on $E$.
\end{prop} 

\begin{proof}
    See \cite[Proposition 1.5.14]{oxley}.
\end{proof}

The operation of declaring a circuit-hyperplane to be a basis is known in the literature by the name of \emph{relaxation}. Many famous matroids arise as a result of applying this operation to another matroid. For instance the Non-Pappus matroid is the result of relaxing a circuit-hyperplane on the Pappus matroid, and analogously the Non-Fano matroid can be obtained by a relaxation of the Fano matroid.

\begin{ex}\label{example:wheels}
    For $k\geq 2$, consider the graph $C_k$ given by a cycle of length $k$ (when $k=2$ this amounts to two parallel edges connecting a pair of points). The cone over $C_k$ gives a graph that we call the \emph{$k$-wheel}. The underlying (graphic) matroid of rank $k$ is denoted by $\W_k$ and will also be called the $k$-wheel. Figure \ref{fig:wheel} depicts the 5-wheel.
    
    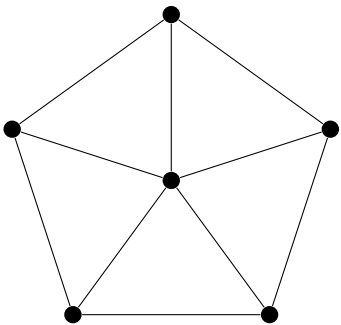
\begin{figure}[h]
        \centering
        \begin{tikzpicture}  
    	[scale=1.7,auto=center,every node/.style={circle, fill=black, inner sep=2.3pt}] 
    	\tikzstyle{edges} = [thick];
    	
        \graph  [empty nodes, clockwise, radius=1em,
        n=9, p=0.3] 
            { subgraph C_n [n=5,m=3,clockwise,radius=2.2cm,name=A]-- mid};
            \foreach \i [count=\xi from 2]  in {1,2,...,5}{
            \node[scale=0.1,fill=black, inner sep=2.3pt] at (A \i){\xi}; }
    \end{tikzpicture}\caption{The wheel graph $\W_5$}\label{fig:wheel}
    \end{figure}
    
    The subset $C_k$ corresponds to a circuit of rank $k-1$ in the matroid $\W_k$. This circuit is also a flat since adding any other edge increases the rank, therefore it is also a hyperplane. The relaxation of this circuit-hyperplane yields a (non-graphic) matroid, which will be called the $k$-\emph{whirl} and denoted $\W^k$. The Kazhdan--Lusztig polynomial and $Z$-polynomial of wheels and whirls were computed in \cite[Theorem 1.1 and 1.6]{wheels}. For example, when $k=5$, the polynomials are:
    \begin{align*}
        P_{\W_5}(t) &= 1+11t+5t^2,\\
        P_{\W^5}(t) &= 1+15t+10t^2,\\
        Z_{\W_5}(t) &= 1 +21t +80t^2 +80t^3 +21t^4 +t^5,\\
        Z_{\W^5}(t) &= 1 +25t +100t^2 +100t^3 +25t^4 +t^5.
    \end{align*}
\end{ex}

\subsection{Structural properties of relaxations}

Our first step towards a proof of our main results is characterizing the flats of a relaxed matroid $\widetilde{\M}$ in terms of the flats of $\M$.

\begin{prop}\label{flatsrelaxation}
	Let $\M=(E,\mathscr{B})$ be a matroid having a circuit-hyperplane $H$. Let $\widetilde{\M}$ be the relaxed matroid. Then, the set of flats $\widetilde{\mathscr{F}}$ of $\widetilde{\M}$ is given by:
	\[\widetilde{\mathscr{F}} = \left(\mathscr{F}\smallsetminus\{H\}\right) \cup \{A\subseteq H : |A| = |H|-1\},\]
	where $\mathscr{F}$ is the set of flats of $\M$.
\end{prop}

\begin{proof}
	Notice that the rank function $\widetilde{\rk}$ of $\widetilde{\M}$ coincides with the rank function $\rk$ of $\M$ with the only exception of $\rk(H)+1=\widetilde{\rk}(H)$.
	
	Let $F$ be a flat of $\widetilde{\M}$ that is not a flat of $\M$. Then $\widetilde{\rk}(F\cup \{e\})>\widetilde{\rk}(F)$ for all $e\notin F$. Since $F\neq H$, we have that $\widetilde{\rk}(F)=\rk(F)$. Notice that there exists an $e$ such that $F\cup \{e\}=H$, since otherwise our inequality holds for all $e$ with $\rk$ instead of $\widetilde{\rk}$ and thus contradicts that $F$ is not a flat of $\M$. Then $F\subseteq H$ and $|F|=|H|-1$, as claimed.
	
	The reverse inclusion follows from the fact that all such sets are flats of $\widetilde{\M}$. If $F\neq H$ is a flat of $\M$, then $\rk(F\cup \{e\}) > \rk(F)$ for all $e\notin F$; in particular, by using $\widetilde{\rk}$ instead of $\rk$ this will still be true even if $F\cup \{e\}=H$, because in that case $\widetilde{\rk}(F\cup \{e\}) = 1 + \rk(F\cup \{e\}) > \rk(F) = \widetilde{\rk}(F)$. Also, if $A = H\smallsetminus \{h\}$ for some $h\in H$, then clearly adding $h$ to $A$ will increase its rank in $\widetilde{\M}$, so let us pick an element $e\notin H$. By applying the submodular inequality of the rank and using that $H$ is a circuit and a hyperplane, we obtain:
	    \[ \underbrace{\rk(H)}_{\rk(\M)-1} + \rk((H\smallsetminus \{h\})\cup\{e\}) \geq \underbrace{\rk(H\smallsetminus \{h\})}_{\rk(\M)-1} + \underbrace{\rk(H\cup\{e\})}_{\rk(\M)}.\]
	This gives us that $\rk((H\smallsetminus \{h\})\cup\{e\})=\rk(\M)$ or, in other words, that $(H\smallsetminus\{h\})\cup\{e\}$ is a basis. Also, notice that as $A\cup\{e\} \neq H$, we have: 
        \[\widetilde{\rk}(A\cup\{e\}) = \rk(A\cup \{e\}) = \rk((H\smallsetminus\{h\})\cup\{e\}) = \rk(\M).\]
	Since $\rk(\M) = \widetilde{\rk}(H) = \widetilde{\rk}(A) + 1$, it follows that $A$ is indeed a flat of $\widetilde{\M}$.
\end{proof}

\begin{obs}\label{tuttecharacteristic}
    Observe that, as we mentioned in the first paragraph of the preceding proof, if $\M$ is a matroid of rank $k$ that has a circuit-hyperplane $H$, and $\widetilde{\M}$ is the corresponding relaxation, the rank functions $\operatorname{rk}$ of $\M$ and $\widetilde{\operatorname{rk}}$ of $\widetilde{\M}$ coincide everywhere but in $H$. Therefore, if we consider the Tutte polynomial of $\widetilde{\M}$, we have:
    \begin{align*} 
        T_{\widetilde{\M}}(x,y) &= \sum_{A\subseteq E} (x-1)^{\widetilde{\rk}(E) - \widetilde{\rk}(A)}(y-1)^{|A|-\widetilde{\rk}(A)}\\
        &= (x-1)^{\widetilde{\rk}(E) - \widetilde{\rk}(H)} (y-1)^{|H|-\widetilde{\rk}(H)} + \sum_{\substack{A\subseteq E\\A\neq H}} (x-1)^{\widetilde{\rk}(E) - \widetilde{\rk}(A)}(y-1)^{|A|-\widetilde{\rk}(A)}\\
        &= (x-1)^{k-k} (y-1)^{k - k} + \sum_{\substack{A\subseteq E\\A\neq H}} (x-1)^{\rk(E) - \rk(A)} (y-1)^{|A|-\rk(A)}\\
        &= 1 + T_{\M}(x,y) - (x-1)^{\rk(E) - \rk(H)}(y-1)^{|H|-\rk(H)}\\
        &= 1 + T_{\M}(x,y) - (x-1)^{1}(y-1)^{1}\\
        &= T_{\M}(x,y) - xy + x + y.
    \end{align*}
    In particular, since the characteristic polynomial of a matroid can be obtained by specifying the Tutte polynomial as follows, $\chi_{\M}(t) = (-1)^kT_{\M}(1-t,0)$, we obtain:
        \[ \chi_{\widetilde{\M}}(t) - \chi_{\M}(t) = (-1)^k (1-t),\]
    which gives a relation between the characteristic polynomial of a matroid and that of a relaxation.
\end{obs}

\begin{oss}\label{rank1}
    Assume that $\M$ is a matroid of rank $1$ having a circuit-hyperplane $H$. We must have that $H$ is a singleton. Moreover, since $H$ is a hyperplane, adding any other element to it should increase the rank. In other words, we have that $H = \{h\}$ and $h$ is the \emph{only} loop of $\M$. Conversely, if a matroid $\M$ of rank $1$ has a unique loop $h$, then $H = \{h\}$ is a circuit-hyperplane. The matroids of rank $1$ having one loop are isomorphic to direct sums of the form $\U_{0,1}\oplus \U_{1,n-1}$, and their relaxations are isomorphic to $\U_{1,n}$. Moreover, for matroids of rank strictly greater than $1$, the presence of a circuit-hyperplane implies automatically that the matroid is loopless; this is because if $\M$ has a loop, it belongs to every flat of $\M$, preventing any hyperplane from being a circuit. 
\end{oss}

\subsection{Relaxations in the Kazhdan--Lusztig framework}

Let us now state a result relating all the Kazhdan--Lusztig invariants of a matroid and those of a relaxation.

\begin{teo}\label{main}
    For each integer $k\geq 1$ there exist polynomials $p_k$, $q_k$ and $z_k$ with integer coefficients such that for every matroid $\M$ of rank $k$ having a circuit-hyperplane $H$ the following equalities hold:
        \begin{align*} 
            P_{\widetilde{\M}}(t) &= P_{\M}(t) + p_k(t),\\
            Q_{\widetilde{\M}}(t) &= Q_{\M}(t) + q_k(t),\\
            Z_{\widetilde{\M}}(t) &= Z_{\M}(t) + z_k(t),
        \end{align*}
    where $\widetilde{\M}$ denotes the corresponding relaxation of $\M$ by $H$.
\end{teo}

\begin{proof}
    We shall split the proof into three parts: one for $P_{\M}$, another for $Q_{\M}$ and lastly one for $Z_{\M}$.
    
    Let us prove by induction on $k$ that there exists a unique polynomial $p_k$ with the property of the statement. From Remark \ref{rank1} we already know that $p_1(t) = 1$ does satisfy the property of the statement. Every loopless matroid of rank $2$ has a constant Kazhdan--Lusztig polynomial equal to $1$, hence $p_2(t)=0$ satisfies the property of the statement as well. Now, assume that $k\geq 3$. Using the defining relation for $P_{\M}(t)$, we write 
    \[
    t^kP_{\widetilde{\M}}(t^{-1}) - P_{\widetilde{\M}}(t) = \sum_{\substack{F\in \mathscr{L}(\widetilde{\M})\\F\neq \varnothing}}\chi_{\widetilde{\M}^F}(t)P_{\widetilde{\M}_F}(t)
    \]
    and
    \[
    t^kP_{\M}(t^{-1}) - P_{\M}(t) = \sum_{\substack{F\in\mathscr{L}(\M)\\F\neq \varnothing}}\chi_{\M^F}(t)P_{\M_F}(t).
    \]
    
    Subtracting the second equation from the first we obtain, using Proposition \ref{flatsrelaxation}:
    \begin{align}
    &t^k\left[P_{\widetilde{\M}}(t^{-1}) - P_{\M}(t^{-1})\right] - \left[P_{\widetilde{\M}}(t) - P_{\M}(t) \right] =\nonumber\\
    &\sum_{\substack{A\subseteq H\\|A|=k-1}}\chi_{\widetilde{\M}^A}(t)P_{\widetilde{\M}_A}(t)\\
    &- \chi_{\M^H}(t)P_{\M_H}(t) \\
    &+ \chi_{\widetilde{\M}}(t) - \chi_{\M}(t)\\
    &+ \sum_{\substack{F\in\mathscr{L}(\M)\\ F\neq \varnothing,H,E}} \left(\chi_{\widetilde{\M}^F}(t) P_{\widetilde{\M}_F}(t) - \chi_{\M^F}(t)P_{\M_F}(t)\right).
    \end{align}
    We show that each summand of this expression is independent of $\M$.
    \begin{enumerate}
        \item For all the terms in the first sum, we have $P_{\widetilde{\M}_A}(t) = 1$, since $\widetilde{\M}_A$ is a matroid of rank 1, and $\chi_{\widetilde{\M}^A} = (t-1)^{k-1}$, since $A$ is independent and, thus, $\widetilde{\M}^A \cong \B_{|A|}$. 
        \item Similarly, $P_{\M_H}(t) = 1$ and $\M^H\cong \U_{k-1,k}$, since $H$ is a circuit of rank $k-1$ in $\M$.
        \item By Remark \ref{tuttecharacteristic}, we can write $\chi_{\widetilde{\M}}(t) - \chi_{\M}(t) = (-1)^k(1-t)$.
        \item In the last sum, we observe that if $F\neq H$ and $F\neq E$, then $\widetilde{\M}^F = \M^F$. We then consider two cases.
        \begin{itemize}
            \item[{\tiny $\blacktriangleright$}] If $F\not\subseteq H$, we observe that $\widetilde{\M}_F = \M_F$, thus all the corresponding terms vanish.
            \item[{\tiny $\blacktriangleright$}] If $F\subseteq H$, $\M^F\cong \B_{\rk(F)}$ and $\widetilde{\M}_F = \widetilde{\M_F}$, where $\widetilde{\M_F}$ is the relaxation of $H\smallsetminus F$ in $\M_F$ (here we are using \cite[Proposition 3.3.5]{oxley}). Hence, the sum can be rewritten as
            \[
            \sum_{\varnothing \subsetneq F \subsetneq H} \chi_{\M^F}(t)p_{k-\rk(F)}(t) = \sum_{j=1}^{k-2}\binom{k}{j}(t-1)^jp_{k-j}(t) = \sum_{j=2}^{k-1}\binom{k}{j}(t-1)^{k-j}p_j(t).
            \]
        \end{itemize}
    \end{enumerate} 
    
    By the induction hypothesis we have that the right-hand-side of the preceding equation does not depend on the matroid $\M$ and only depends on the rank $k$. Also, observe that the following becomes a defining equation for $p_k(t)$,
    \[
    t^kp_k(t^{-1}) - p_k(t) = k(t-1)^{k-1} - \chi_{\U_{k-1,k}}(t) + (-1)^k(1-t) + \sum_{j=2}^{k-1}\binom{k}{j}(t-1)^{k-j}p_j(t).
    \]
    Since $p_k$ must have degree strictly smaller than $\frac{k}{2}$, the above equation is satisfied by one and only one polynomial $p_k$. By induction we conclude the proof of the existence and uniqueness of $p_k(t)$ for every $k\geq 1$.
    
    Now let us focus on $Q_{\M}(t)$. When $k=1$, again by Remark \ref{rank1} we have $q_1(t)=1$. On the other hand, for $k\geq 2$, we first write
    \[
    (-t)^kQ_{\widetilde{\M}}(t^{-1}) =  (-1)^kQ_{\widetilde{\M}}(t)
    + \sum_{\substack{F\in\mathscr{L}(\widetilde{\M})\\F\neq E}}(-1)^{\rk(F)}Q_{\widetilde{\M}^F}(t)t^{k-\rk(F)}
    \chi_{\widetilde{\M}_F}(t^{-1}),
    \]
    and a counterpart equality, using $Q_{\M}$ instead of $Q_{\widetilde{\M}}$. Subtracting one equation from the other gives us
    \begin{align*}
        & (-t)^k\left[Q_{\widetilde{\M}}(t^{-1}) - Q_{\M}(t^{-1}) \right] - (-1)^k\left[Q_{\widetilde{\M}}(t) - Q_{\M}(t) \right]
        =\\
        & \sum_{\substack{A\subseteq H \\ |A| = |H| -1}}(-1)^{\rk(A)}Q_{\widetilde{\M}^A}(t)t^{k-\rk(A)}\chi_{\widetilde{\M}_A}(t^{-1})\\
        & - (-1)^{\rk(H)}Q_{\M^H}(t)t^{k-\rk(H)}\chi_{\M_H}(t^{-1})\\
        & + t^k\left[\chi_{\widetilde{\M}}(t^{-1}) - \chi_{\M}(t^{-1}) \right] \\
        & + \sum_{\substack{F\in \mathscr{L}(\M)\\F\neq \varnothing,H,E}}(-1)^{\rk(F)}t^{k-\rk(F)}\left[Q_{\widetilde{\M}^F}(t)\chi_{\widetilde{\M}_F}(t^{-1}) - Q_{\M^F}(t)\chi_{\M_F}(t^{-1})\right].
    \end{align*}
    This permits us to write
    \[
    t^kq_k(t^{-1}) - q_k(t) = (t-1)\left( t^{k-1} +k - Q_{\U_{k-1,k}}(t)+\sum_{j=2}^{k-1}\binom{k}{j}t^{j-1}\right).
    \]
    By an analogous reasoning we conclude the existence and uniqueness of $q_k(t)$ for each $k\geq 1$.
    
    Finally, for the $Z$-polynomial we write
    \[
    Z_{\widetilde{\M}}(t) = \sum_{\substack{F\in\mathscr{L}(\widetilde{\M})}}t^{\rk(F)}P_{\widetilde{\M}_F}(t),
    \]
    and similarly for $Z_{\M}(t)$. Then,
    \begin{align*}
    Z_{\widetilde{\M}}(t) - Z_{\M}(t) &=\sum_{\substack{A\subseteq H\\ |A|=|H|-1}}t^{k-1}P_{\widetilde{\M}_A}(t) - t^{k -1}P_{\widetilde{\M}_H}(t) + \sum_{F\in \mathscr{L}(\M)} t^{\rk(F)}\left(P_{\widetilde{\M}_F}(t) - P_{\M_F}(t) \right)\\
    &= (k-1)t^{k-1} +  \sum_{F\in \mathscr{L}(\M)} t^{\rk(F)}\left(P_{\widetilde{\M}_F}(t) - P_{\M_F}(t) \right).
    \end{align*}
    Considering separately the two cases $F\subseteq H$ and $F\nsubseteq H$ as we did for $P_{\M}$, this can be further simplified to get
    \begin{equation*}\label{zeq}
    z_k(t) = (k-1)t^{k-1} + \sum_{j=2}^k \binom{k}{j}t^{k-j}p_j(t),\end{equation*}
    which completes the proof.
\end{proof}

Now that we know that each of these polynomials is independent of the considered matroids, we will use this fact to give an explicit formula for each of them. To this end, we will consider a case in which both the matroid $\M$ and its relaxation $\widetilde{\M}$ have nice Kazhdan--Lusztig polynomials.

\begin{prop}\label{minimal}
    Let $\M$ be the matroid of rank $k$ and cardinality $n\geq k+1$ given by:
        \[ \M =  \U_{k-1,k}\oplus \U_{1,n-k}.\]
    Then, $\M$ has a circuit-hyperplane $H$. The corresponding relaxation $\widetilde{\M}$, denoted by $\T_{k,n}$, is such that its simplification is isomorphic to $\U_{k,k+1}$.
\end{prop}

\begin{proof}
    Since $\M = \U_{k-1,k}\oplus \U_{1,n-k}$, it is clear that $\rk(\M) = k$ and that the cardinality of $\M$ is $n$. Let us suppose that the ground set of $\M$ consists of the set $E=\{x_1,\ldots,x_n\}$, and let us assume that $E_1=\{x_1,\ldots,x_k\}$ is the ground set of $\U_{k-1,k}$ and $E_2=\{x_{k+1},\ldots,x_n\}$ is the ground set of $\U_{1,n-k}$. 
    
    Let us describe all the flats of $\M$ first. By \cite[p. 125]{oxley} (which characterizes the flats of a direct sum of matroids), we have that the flats of $\M$ are exactly the subsets $F\subseteq E$ such that $F\cap E_1$ is a flat in $\U_{k-1,k}$ and $F\cap E_2$ is a flat in $\U_{1,n-k}$. Thus, $F$ is a flat of $\M$ if and only if we have:
        \[ |F\cap E_1| \in \{0, 1,\ldots, k-2, k\} \;\;\text{ and }\;\; |F\cap E_2| \in \{0, n-k\}.\]

    We claim that:
        \[ H = \{x_1,\ldots,x_k\} = E_1\]
    is a circuit-hyperplane in $\M$. Indeed:
    \begin{itemize}
        \item $H$ is a flat, because of the above conditions.
        \item $H$ has rank $k-1$, since $\rk_{\M} H = \rk_{\U_{k-1,k}} H = k-1$.
        \item $H$ is a circuit, because the removal of any member of $H$ yields an independent subset of $\U_{k-1,k}$ which will of course be an independent subset of the matroid $\M$.
    \end{itemize}
    Now that we know that $H$ is a circuit-hyperplane of $\M$ we can consider the relaxation $\widetilde{\M}$ of $\M$ by $H$. Also, we can use Proposition \ref{flatsrelaxation} to characterize all the flats of $\widetilde{\M}$. In fact, we see that $\widetilde{F} \in \widetilde{\mathscr{F}}$ if and only if:
        \[
        |\widetilde{F}\cap E_1| \in \{0, 1,\ldots, k-2,k-1\} \;\;\text{ and }\;\; |\widetilde{F}\cap E_2| = 0, \]
        \[
        \text{ or }|\widetilde{F}\cap E_1| \in \{0, 1,\ldots, k-2,k\} \;\;\text{ and }\;\; |\widetilde{F}\cap E_2| = n-k. \]
       
    Notice that the set $E_2$ is an atom of this lattice of flats. The remaining $k$ atoms are the elements of $E_1$. Moreover, if we label the elements of $E_1$ as $\overline{1},\ldots,\overline{k}$ and label the atom $E_2$ as $\overline{k+1}$, we can construct an order-preserving bijection from the flats of $\widetilde{\M}$ to the family of subsets of $\{\overline{1},\ldots,\overline{k+1}\}$ having cardinality $\{0,\ldots,k-1,k+1\}$. The latter is just isomorphic to the lattice of flats of $\U_{k,k+1}$ which implies that the simplification of $\widetilde{\M}$ is isomorphic to $\U_{k,k+1}$, as desired.
\end{proof}

A useful property of the (inverse) Kazhdan--Lusztig polynomials and the $Z$-polynomials is that they behave particularly well under direct sums of matroids. Namely, in \cite[Proposition 2.7]{eliasproudfoot} Elias et al. proved that
    \begin{equation}\label{eq:direct-sum} P_{\M_1\oplus\M_2}(t) = P_{\M_1}(t) \cdot P_{\M_2}(t).\end{equation}
An analogous statement for the inverse Kazhdan--Lusztig polynomial of a direct sum of matroids was proved in \cite[Lemma 3.1]{gaoinverse}. For the $Z$-polynomial we have as well
    \begin{equation}\label{eq:direct-sum-zeta} Z_{\M_1\oplus\M_2}(t) = Z_{\M_1}(t) \cdot Z_{\M_2}(t).\end{equation}
The proof follows directly from the definition of the $Z$-polynomial and the fact that the contraction of a direct sum of matroids $\M_1$ and $\M_2$ with ground sets $E_1$ and $E_2$, namely, $(\M_1\oplus\M_2)_F$ is exactly the direct sum $(\M_1)_{E_1\cap F} \oplus (\M_2)_{E_2\cap F}$, see \cite[Exercise 4.2.19]{oxley}. This, together with the property of equation \eqref{eq:direct-sum} yields a straightforward proof of \eqref{eq:direct-sum-zeta}.

\begin{cor}\label{formulitas}
    For each $k\geq 1$ we have:
    \begin{enumerate}[(a)]
        \item $p_k(t) = P_{\U_{k,k+1}}(t) - P_{\U_{k-1,k}}(t)$.
        \item $q_k(t) = Q_{\U_{k,k+1}}(t) - Q_{\U_{k-1,k}}(t)$.
        \item $z_k(t) = Z_{\U_{k,k+1}}(t) - (1+t) Z_{\U_{k-1,k}}(t)$.
    \end{enumerate}
\end{cor}

\begin{proof}
    Let us call $\M = \U_{k-1,k}\oplus \U_{1,1}$ (this corresponds to taking $n = k +1$ in the last proposition) and $\widetilde{\M}$ its relaxation. Since $\widetilde{\M}$ and $\U_{k,k+1}$ have isomorphic lattices of flats, we know that $P_{\widetilde{\M}}(t) = P_{\U_{k,k+1}}(t)$. Also:
    \begin{align*}
        p_k(t) &= P_{\widetilde{\M}}(t) - P_{\M}(t)\\
        &= P_{\U_{k,k+1}}(t) - P_{\U_{k-1,k}}(t)\cdot P_{\U_{1,1}}(t)\\
        &= P_{\U_{k,k+1}}(t) - P_{\U_{k-1,k}}(t),
    \end{align*}
    where we used the property of equation \eqref{eq:direct-sum} and the fact that $P_{\U_{1,1}}(t) = 1$. An identical proof but using $Q$ instead of $P$ shows (b). In the case of the $Z$-polynomial, we observe that $Z_{\U_{1,1}}(t) = 1+t$.
\end{proof}

\begin{ej}
    As a simple application we find that our results are consistent with the ones found by Lu et al. in \cite[Theorem 1.1]{wheels} for wheel matroids and whirl matroids (cf. Example \ref{example:wheels}). Since the wheel matroid $\W_k$ is a matroid of rank $k$, using Theorem \ref{main}, we obtain that:
        \begin{align*} 
            p_k(t) &= P_{\W^k}(t) - P_{\W_k}(t),\\
            q_k(t) &= Q_{\W^k}(t) - Q_{\W_k}(t),\\
            z_k(t) &= Z_{\W^k}(t) - Z_{\W_k}(t).
        \end{align*}
    It is possible to prove with a direct computation that $P_{\W^k}(t) - P_{\W_k}(t)$ does indeed coincide with $P_{\U_{k,k+1}}(t) - P_{\U_{k-1,k}}(t)$, as one should expect.
\end{ej}

When the matroid $\M$ under consideration is a uniform matroid of the form $\U_{k,k+1}$, there exist formulas for $P_{\M}$ and $Q_{\M}$; see \cite[Theorem 1.2(1)]{proudfoot2015} and \cite[Theorem 3.3]{gaoinverse} respectively. Also, for wheel and whirl matroids (see Example \ref{example:wheels}), in \cite[Theorem 1.6]{wheels} Lu et al. obtained simple formulas for the $Z$-polynomial. Using all these results, we can deduce many properties about our polynomials $p_k$, $q_k$ and $z_k$.  Whenever we have a polynomial $f(t)$, we will denote by $[t^i]f(t)$ the coefficient of $t^i$ in $f(t)$.

\begin{teo}\label{positivity}
    For each $k\geq 1$, we have that the coefficients of $p_k$, $q_k$ and $z_k$ are all non-negative. Moreover, the following statements hold:
    \[\deg p_k = \deg q_k = \left\lfloor\frac{k-1}{2}\right\rfloor\]
    and 
    \[\deg z_k = k-1.\] 
\end{teo}

\begin{proof}
    We will consider each family of polynomials separately.
    \begin{itemize}
        \item By \cite[Theorem 1.2(1)]{proudfoot2015}, we know that $[t^i]P_{\U_{k,k+1}}$ is equal to $\frac{1}{i+1}\binom{k-i-1}{i}\binom{k+1}{i}$ for each $0\leq i \leq \lfloor\frac{k-1}{2}\rfloor$. When $i\in \left\{0,\ldots, \lfloor\frac{k-2}{2}\rfloor\right\}$ we can compute:
        \[    [t^i]p_k(t) = \frac{1}{i+1}\binom{k-i-1}{i}\binom{k+1}{i} - \frac{1}{i+1}\binom{k-i-2}{i}\binom{k}{i}.
        \]
        To conclude the non-negativity of all the coefficients of $p_k$, it suffices to show that 
        \[ \binom{k-i-1}{i}\binom{k+1}{i} \geq \binom{k-i-2}{i}\binom{k}{i}.\]
        However, expanding the binomial coefficients as quotients of factorials, this is just equivalent to proving that
        \[ \frac{k-i-1}{k-2i-1} \cdot \frac{k+1}{k+1-i} \geq 1,\]
        which is clear, since both factors on the left are greater or equal to $1$. Moreover, the above argument shows that the coefficients of degree $1,\ldots, \lfloor\frac{k-2}{2}\rfloor$ are strictly positive, and that the constant term is zero when $k > 1$. Now, if $k$ is even, we have that $\lfloor\frac{k-1}{2}\rfloor = \lfloor\frac{k-2}{2}\rfloor$, so we conclude that $\deg p_k = \lfloor\frac{k-1}{2}\rfloor$, and if $k$ is odd, then $\deg P_{\U_{k,k+1}} > \deg P_{\U_{k-1,k}}$, so that we clearly have $\deg p_k = \deg P_{\U_{k,k+1}} = \lfloor\frac{k-1}{2}\rfloor$, as desired.
         
        \item By \cite[Theorem 3.3]{gaoinverse}, we know that the coefficient $[t^i]Q_{\U_{k,k+1}}(t)$ is given by:
            \[ [t^i] Q_{\U_{k,k+1}}(t) = (k+1)\binom{k}{i}\cdot \frac{k-2i}{(i+1)(k+1-i)}.\]
        Notice that for $0\leq i \leq \lfloor\frac{k-2}{2}\rfloor$, we have that:
            \[ [t^i] q_k = (k+1)\binom{k}{i}\frac{k-2i}{(i+1)(k+1-i)} - k \binom{k-1}{i}\frac{k-1-2i}{(i+1)(k-i)},\]
        and we want to show that:
            \[ (k+1)\binom{k}{i}\frac{k-2i}{(i+1)(k+1-i)} \geq k \binom{k-1}{i}\cdot \frac{k-1-2i}{(i+1)(k-i)},\]
        which, expanding as quotients of binomial coefficients, is equivalent to:
            \[ \frac{k+1}{k+1-i}\cdot \frac{k-2i}{k-2i-1} \geq 1.\]
        which is true because both fractions are trivially greater or equal than $1$. Moreover, since the second fraction is strictly greater than $1$, we actually have that all the coefficients of $q_k$ are strictly positive. The same considerations regarding the parity of $k$ yield the conclusion about the degree of $q_k$ in an analogous fashion to what was said for $p_k$ before.
        \item The statement is trivially true for $k=1$ as actually one has $z_1(t)=0$. By \cite[Theorem 1.6]{wheels} we know that for $k\geq 2$, 
            \begin{align*}
                Z_{\W^k}(t) &= \sum_{j=0}^k \binom{k}{j}^2 t^j\\
                Z_{\W_k}(t) &= \sum_{j=0}^k\left(\binom{k}{j}^2 - \frac{2}{k}\binom{k}{j+1}\binom{k}{j-1}\right) t^j.
            \end{align*}
        In Example \ref{example:wheels} we proved that $\W^k$ is a relaxation of $\W_k$, so that the polynomial $z_k$ can be obtained by considering the difference of the above two expressions, namely
        \[ z_k(t) =  \sum_{j=1}^{k-1}\frac{2}{k}\binom{k}{j+1}\binom{k}{j-1} t^j,\]
        which has degree $k-1$ and non-negative coefficients. Alternatively, one can use the formulas obtained in \cite[Proposition 4.9]{proudfootzeta} to obtain the exact same expression as above.
            \qedhere
    \end{itemize}
\end{proof}

\begin{table}
	\parbox{.45\linewidth}{\begin{tabular}{>{$}l<{$}|*{6}{c}}
		\multicolumn{1}{l}{$k$} &&&&&\\\cline{1-1}
		1 &1&&&&&\\
		2 &0&&&&&\\
		3 &0&2&&&&\\
		4 &0&3&&&&\\
		5 &0&4&5&&&\\
		6 &0&5&16&&&\\
		7 &0&6&35&14&&\\
		8 &0&7&64&70&&\\
		9 &0&8&105&216&42&\\
		10&0&9&160&525&288&\\
		11&0&10&231&1100&1155&132\\
		12&0&11&320&2079&3520&1155\\\hline
		\multicolumn{1}{l}{} &0&1&2&3&4&5\\\cline{2-7}
		\multicolumn{1}{l}{} &\multicolumn{6}{c}{$i$}
	\end{tabular}
	\caption{$[t^i]p_k(t)$}\label{tablas}}
	\quad\quad\quad
	\parbox{.45\linewidth}{\begin{tabular}{>{$}l<{$}|*{6}{c}}
		\multicolumn{1}{l}{$k$} &&&&&\\\cline{1-1} 
		1 &1&&&&&\\
		2 &1&&&&&\\
		3 &1&2&&&&\\
		4 &1&3&&&&\\
		5 &1&4&5&&&\\
		6 &1&5&9&&&\\
		7 &1&6&14&14&&\\
		8 &1&7&20&28&&\\
		9 &1&8&27&48&42&\\
		10&1&9&35&75&90&\\
		11&1&10&44&110&165&132\\
		12&1&11&54&154&275&297\\\hline
		\multicolumn{1}{l}{} &0&1&2&3&4&5\\\cline{2-7}
		\multicolumn{1}{l}{} &\multicolumn{6}{c}{$i$}
	\end{tabular}
	\caption{$[t^i]q_k(t)$}\label{tablas2}}
	\quad\quad\quad
	\parbox{.45\linewidth}{\begin{tabular}{>{$}l<{$}|*{10}{c}}
		\multicolumn{1}{l}{$k$} &&&&&\\\cline{1-1}
		1 &0&&&&&\\
		2 &0&1&&&&\\
		3 &0&2&2&&&\\
		4 &0&3&8&3&&\\
		5 &0&4&20&20&4&\\
		6 &0&5&40&75&40&5\\
		7 &0&6&70&210&210&70&6\\
		8 &0&7&112&490&784&490&112&7\\
		9 &0&8&168&1008&2352&2352&1008&168&8\\
		10&0&9&240&1890&6048&8820&6048&1890&240&9\\\hline
		\multicolumn{1}{l}{} &0&1&2&3&4&5&6&7&8&9\\\cline{2-11}
		\multicolumn{1}{l}{} &\multicolumn{10}{c}{$i$}
	\end{tabular}
	\caption{$[t^i]z_k(t)$}\label{tablas3}}
\end{table}

\begin{obs}
    In Table \ref{tablas} we can see the coefficients of $p_k$ for small values of $k$. In Table \ref{tablas2} we see the corresponding coefficients for $q_k$. In Table \ref{tablas3} we see the coefficients of $z_k$. According to the \emph{Online Encyclopedia of Integer Sequences} (OEIS) \cite{sloane}, the coefficients of $q_k$ coincide with the sequence \href{http://oeis.org/A008315}{A008315}, and the coefficients of $z_k$ coincide with the sequence \href{http://oeis.org/A145596}{A145596}.
\end{obs}

\subsection{A polyhedral interpretation}

We end this section with a brief and informal discussion of what is going on at the level of matroid polytopes. Recall that if $\M$ is a matroid with ground set $E=\{1,\ldots,n\}$ and set of bases $\mathscr{B}$, its \emph{base polytope} is defined as the convex hull in $\mathbb{R}^n$ of the indicator vectors of all the bases of $\M$. In other words, if for each basis $B\in\mathscr{B}$ we denote $e_B = \sum_{i\in B} e_i$, where $e_i$ is the $i$-th canonical vector of $\mathbb{R}^n$, then the base polytope of $\M$ is defined as
    \[ \mathscr{P}(\M) = \text{convex hull} \{e_B: B\in \mathscr{B}(\M)\}.\]

In \cite[Theorems 8.8 and 8.9]{ardila}, Ardila and Sanchez proved that the (inverse) Kazhdan--Lusztig polynomial of a matroid is a ``valuative invariant'' for matroidal subdivision of base polytopes. This means that if one picks a subdivision of $\mathscr{P}(\M)$ into subpolytopes that are themselves the base polytopes of suitable matroids, then one can compute the (inverse) Kazhdan--Lusztig polynomial of the original matroid by first computing it for the smaller pieces and then using an inclusion-exclusion argument. For a precise account and thorough discussion of the properties of valuative invariants of (poly)matroids we refer the reader to \cite{derksenfink} and \cite{ardila-fink-rincon}. 

On the other hand, in \cite[Theorem 5.4]{ferroni2} Ferroni proved that the circuit-hyperplane relaxation consists geometrically of stacking the polytope of a ``minimal matroid'' on a facet of the polytope. In other words, using the notation that we have established so far, we have that for every matroid $\M$ of rank $k$, cardinality $n$, having a circuit-hyperplane $H$, there is a decomposition of the polytope:
    \[ \mathscr{P}(\widetilde{\M}) = \mathscr{P}(\M) \cup \mathscr{P}(\T_{k,n}),\]
where $\T_{k,n}$ is the relaxed matroid of Proposition \ref{minimal} (i.e. the relaxation of $\U_{k-1,k}\oplus \U_{1,n-k}$).

If we denote the matroid with base polytope $\mathscr{P}(\M)\cap \mathscr{P}(\T_{k,n})$ by $\N$, what we have proved in the present article is that:
    \[ P_{\widetilde{\M}}(t) = P_{\M}(t) + P_{\T_{k,n}}(t) - P_\N(t),\]
which is something that we might have expected from the valuative property of the Kazhdan--Lusztig polynomials on the polytopes. Observe that by Proposition \ref{minimal} what we have is that $P_{\T_{k,n}}(t) = P_{\U_{k,k+1}}(t)$. Also, from the proof of Theorem 5.4 in \cite{ferroni2} one can deduce that $\N\cong \U_{k-1, k}\oplus \U_{1,n-k}$, so that $P_\N(t) = P_{\U_{k-1,k}}(t)$, and we recover what we had initially expected:
    \[ P_{\widetilde{\M}}(t) - P_{\M}(t) = p_k(t),\]
and the polynomial on the right does not depend on $\M$ but only on its rank. Observe that a similar phenomenon occurs for the Tutte polynomial, as we proved in Remark \ref{tuttecharacteristic}. This is not a coincidence, because the Tutte polynomial is also a valuative invariant. 

\begin{obs}
    It has not been proved yet that $Z$-polynomials are valuative invariants, but we expect that to be true and, along these lines, our Theorem \ref{main} is a good evidence to support that assertion. We believe that the techniques used by Ardila and Sanchez in \cite{ardila} can be extended to the $Z$-polynomials as well.
\end{obs}

\section{Free bases and non-degeneracy}\label{sec:four}

\subsection{Relaxations and free bases}

The fact that $\M$ has a circuit-hyperplane can be reformulated as a property of $\widetilde{\M}$ as follows.

\begin{defi}
    Let $\M$ be a matroid on $E$ that has at least two bases. We say that a basis $B$ of $\M$ is \emph{free} if $B\cup \{e\}$ is a circuit for each $e\in E\smallsetminus B$.
\end{defi}

\begin{lema}
    Let $\M$ be a matroid with a circuit-hyperplane $H$. Then, the corresponding relaxation $\widetilde{\M}$ is such that its basis $H$ is free. Conversely, if a matroid $\M$ with set of bases $\mathscr{B}$, $|\mathscr{B}|\geq 2$, has a free basis $B$, then $\mathscr{B}\smallsetminus \{B\}$ is the set of bases of a matroid $\M'$ that has $B$ as a circuit-hyperplane.
\end{lema}

\begin{proof}
    Let us assume that $\M$ has ground set $E=\{x_1,\ldots,x_n\}$, and that $H=\{x_1,\ldots,x_k\}$ is a circuit-hyperplane of $\M$. Consider the relaxation $\widetilde{\M}$ of $\M$ by $H$. Pick an element $x_j\in E\smallsetminus H = \{x_{k+1},\ldots, x_n\}$. We want to prove that $H\cup \{x_j\}$ is a circuit in $\widetilde{\M}$. It is clear that $H\cup \{x_j\}$ is dependent in $\widetilde{\M}$. Let us analyze what happens if we remove one element from $H\cup\{x_j\}$.
    \begin{itemize}
        \item If we remove $x_j$, we recover the basis $H$, which is of course independent in $\widetilde{\M}$.
        \item If we fix an index $1\leq i\leq k$ and consider the set $A=(H\cup\{x_j\})\smallsetminus \{x_i\}$. It is clear that $|A| = k$, and we can rewrite $A=(H\smallsetminus \{x_i\}) \cup \{x_j\}$. Since $H$ was a circuit in $\M$, then we have that $H\smallsetminus\{x_i\}$ is an independent subset of $\M$ of rank $k-1$, and hence an independent subset of $\widetilde{\M}$ of rank $k-1$. Since we proved in Proposition \ref{flatsrelaxation} that $H\smallsetminus\{x_i\}$ is a flat of $\widetilde{\M}$, in particular adding any element to this set increases its rank. Thus, $A=(H\smallsetminus\{x_i\})\cup\{x_j\}$ has rank $k$ in $\widetilde{\M}$. Since its cardinality is also $k$, we get that it is a basis of $\widetilde{\M}$ and, of course, independent.
    \end{itemize}
    It follows that $H\cup\{e\}$ is a circuit of $\widetilde{\M}$ for all $e\in E\smallsetminus H$, and hence $H$ is a free basis of $\widetilde{\M}$.
    
    Conversely, let us assume that the ground set of $\M$ is $E=\{x_1,\ldots,x_n\}$ and that $B = \{x_1,\ldots,x_k\}$ is a free basis of $\M$. Since $B\cup \{x_j\}$ is a circuit for every $j\in \{k+1,\ldots,n\}$, we get that:
        \begin{equation}\label{exprop} B_{ij} := (B\cup \{x_j\})\smallsetminus \{x_i\} = (B\smallsetminus\{x_i\})\cup \{x_j\}\end{equation}
    is a basis of $\M$ for every $1\leq i\leq k$ and $k+1\leq j\leq n$.
    
    Now, notice that since $|\mathscr{B}|\geq 2$, we automatically have that $\mathscr{B}\smallsetminus \{B\}\neq \varnothing$. So, to prove that $\mathscr{B}\smallsetminus\{B\}$ is the set of bases of a matroid, we must show that the basis-exchange-property holds, that is, for every two distinct bases of $\M$ that differ from $B$, and $a\in B_1\smallsetminus B_2$, there exists some $b\in B_2\smallsetminus B_1$ such that $(B_1\smallsetminus\{a\})\cup\{b\}\in \mathscr{B}\smallsetminus\{B\}$. 
    
    Let us start applying the basis-exchange-property to the bases $B_1$ and $B_2$ of $\M$ and $a\in B_1\smallsetminus B_2$, which yields an element $b'\in B_2\smallsetminus B_1$ such that $(B_1\smallsetminus\{a\})\cup \{b'\} \in \mathscr{B}$. If $(B_1\smallsetminus\{a\})\cup \{b'\}\neq B$ then the result follows. Let us assume then that $(B_1\smallsetminus\{a\})\cup\{b'\} = B$. It follows that:
        \[ B_1 \smallsetminus \{a\} = B\smallsetminus \{b'\},\]
    and since $a\in B_1$, the cardinality of the set on the left is strictly less than $|B_1| = |B|$, from where it follows that $b'\in B$. Hence $b'=x_i$ for some $1\leq i\leq k$. From equation \eqref{exprop}, we see that choosing any element $b\in B_2\smallsetminus B$ (there is at least one, since $B_2$ was assumed to be different from $B$), we have that $(B\smallsetminus \{b'\})\cup \{b\}$ is a basis of $\M$. Hence:
        \[ (B_1 \smallsetminus \{a\})\cup \{b\} \in \mathscr{B}\smallsetminus \{B\},\]
    and this proves that $\mathscr{B}\smallsetminus\{B\}$ is the set of bases of a matroid that we will denote by $\M'$. 
    
    To finish the proof, it remains to prove that $B$ is a circuit-hyperplane of $\M'$. 
        \begin{itemize}
            \item Let us prove that $B$ is a circuit of $\M'$. Consider $B\smallsetminus\{x_i\}$ for $1\leq i\leq k$. This is independent, since equation \eqref{exprop} guarantees that it is contained in another basis of $\M$ different from $B$, and hence in a basis of $\M'$.
            \item Let us prove that $B$ is a hyperplane. Again, equation \eqref{exprop} yields immediately that the rank of $B$ in $\M'$ is $k-1$. Finally, $B$ is a flat, since adding any element to $B$ would increase its rank again by \eqref{exprop}.\qedhere
        \end{itemize}
\end{proof}

\begin{obs}
    Let us fix a matroid $\M$ on $E$ with set of bases $\mathscr{B}$. We know \emph{one} particular scenario in which adding a new member to $\mathscr{B}$ preserves the basis-exchange-property. It is a result by Truemper \cite[Lemma 6]{truemper} (see also \cite{mills}) that if $\M$ is a simple and cosimple matroid on $E$ and $A\subseteq E$ is such that $\mathscr{B}(\M)\cup \{A\}$ is the set of bases of a matroid, then $A$ \emph{has to be} a circuit-hyperplane. Analogously, under the simplicity and cosimplicity hypotheses, the only case in which the removal of one basis yields a set of bases of a matroid is when the considered basis is free.
\end{obs}

\subsection{The proof of Theorem \ref{mainresult1}}

Now we can relate a combinatorial property of a basis of $\M$ to the Kazhdan--Lusztig theory of $\M$. 

\begin{teo}\label{freebasisnondeg}
    If $\M$ has a free basis, then $\M$ is non-degenerate.
\end{teo}

\begin{proof}
    If $\M$ has a free basis $B$, then $\M$ is the relaxation of the matroid $\M'$ with set of bases $\mathscr{B}(\M)\smallsetminus \{B\}$, using the circuit-hyperplane $B$. Assume that the rank of $\M$ and $\M'$ is $k$. Using Theorem \ref{main}, we obtain that:
        \[ P_{\M}(t) = P_{\M'}(t) + p_k(t).\]
    Since the coefficients of $P_{\M'}$ are non-negative \cite[Theorem 1.2]{bradenhuh}, using the fact that the coefficients of $p_k$ are also non-negative and that the degree of $p_k$ is $\left\lfloor\frac{k-1}{2}\right\rfloor$, it follows that
        \[ \deg P_{\M} = \left\lfloor \frac{k-1}{2}\right\rfloor,\]
    and thus $\M$ is non-degenerate.
\end{proof}

\begin{obs}
    We believe that asymptotically all matroids have a free basis. This belief is supported by an equivalent conjecture posed by Bansal et al. \cite[Conjecture 22]{bansal}. A proof of the aforementioned conjecture would automatically imply that asymptotically all matroids are non-degenerate.
\end{obs}

\begin{obs}
    We do not need the non-negativity of the middle coefficients of the Kazhdan--Lusztig polynomials of matroids. In the proof of Theorem \ref{freebasisnondeg} what we need is that the coefficient of degree $\lfloor \frac{k-1}{2}\rfloor$ is non-negative. Although Braden et al. proved that it is indeed $\geq 0$, we are not aware of a proof of the non-negativity of that coefficient alone.
\end{obs}

\section{An application to sparse paving matroids}\label{sec:five}

\subsection{An overview} 

A matroid $\M$ of rank $k$ is said to be \emph{paving} if all of its circuits have size at least $k$. A paving matroid $\M$ is said to be \emph{sparse paving} if all of its hyperplanes have size at most $k$. The class of all sparse paving matroids is particularly relevant in the framework of asymptotic matroid enumeration. If we denote by $\operatorname{mat}(n)$ the number of matroids with ground set $\{1,\ldots,n\}$, and $\operatorname{sp}(n)$ the number of sparse paving matroids among them, a conjecture posed by Mayhew et al. in \cite{mayhew} asserts that
    \[ \lim_{n\to\infty} \frac{\operatorname{sp}(n)}{\operatorname{mat}(n)} = 1.\]
(see \cite[Conjecture 1.6]{mayhew} and the discussion after that statement). In other words, they conjecture that sparse paving matroids are predominant.

As was mentioned in the introduction, in \cite[Theorem 1]{lee} Lee, Nasr and Radcliffe found a formula for the Kazhdan--Lusztig polynomial of sparse paving matroids by enumerating suitable (skew) Young tableaux. They took advantage of their formula and several inequalities to prove that the Kazhdan--Lusztig polynomial $P_{\M}$ of a sparse paving matroid $\M$ of rank $k$ and cardinality $n$ satisfies two properties:
    \begin{itemize}
        \item The polynomial $P_{\M}$ is coefficient-wisely smaller than $P_{\U_{k,n}}$.
        \item The coefficients of $P_{\M}$ are non-negative.
    \end{itemize}

It is possible to use our results to give an independent proof of these facts, and also prove identical statements for $Q_{\M}$ and $Z_{\M}$ in a more direct way. We recall that in \cite{bradenhuh} it has been proved that $P_{\M}$ and $Q_{\M}$ have non-negative coefficients for \emph{all} matroids. The general proof is much more involved than what we are about to do here, but we include the one for $Q_{\M}$ and $Z_{\M}$ when $\M$ is sparse paving as applications of our results.

We will leverage the following easy result, which gives an alternative characterization of sparse paving matroids.

\begin{lema}\label{sparsepavingdefi}
    A matroid $\M$ on $E$ of rank $k$ is sparse paving if and only if every subset $A\subseteq E$ of cardinality $k$ is either a basis or a circuit-hyperplane.
\end{lema}

\begin{proof}
    If a matroid is such that every subset of cardinality $k$ is either a basis or a circuit-hyperplane, then it automatically is sparse paving. This is because the existence of a circuit of size less than $k$ is ruled out. Such a circuit can be completed to a set of cardinality $k$ which will fail to be a basis or a circuit-hyperplane. Analogous considerations avoid the possibility of the existence of a hyperplane of size larger than $k$. Conversely, assume that $\M$ is sparse paving and pick a subset $A$ of cardinality $k$ that is not a basis. Hence, we have that $A$ is dependent, and thus contains a circuit $C$. Since $\M$ is paving we have that $k\leq |C| \leq |A| = k$, and since $C\subseteq A$, it follows that $C=A$ and hence $A$ is a circuit. Proceeding in the opposite way, we can prove that $A$ is also a hyperplane.
\end{proof}

It is immediate that uniform matroids are sparse paving and, moreover, that all sparse paving matroids can be relaxed to a uniform matroid. This is because after relaxing one circuit-hyperplane, the remaining circuit-hyperplanes are still circuit-hyperplanes in the new matroid.

\subsection{The Kazhdan--Lusztig theory of sparse paving matroids} 

By combining the observations in the prior paragraph with our Theorem \ref{mainresult2} we get the following result.

\begin{teo}\label{sparsepaving}
    Let $\M$ be a sparse paving matroid of rank $k$ and cardinality $n$. Assume that $\M$ has exactly $\lambda$ circuit-hyperplanes. Then:
    \begin{align*}
        P_{\M}(t) &= P_{\U_{k,n}}(t) - \lambda p_k(t),\\
        Q_{\M}(t) &= Q_{\U_{k,n}}(t) - \lambda q_k(t),\\
        Z_{\M}(t) &= Z_{\U_{k,n}}(t) - \lambda z_k(t).
    \end{align*}
\end{teo}

\begin{proof}
    If $\M$ has $\lambda$ circuit-hyperplanes, after relaxing all of them what we end up obtaining is just $\U_{k,n}$. Hence:
        \[ P_{\U_{k,n}}(t) = P_{\M}(t) + \underbrace{p_k(t) + \cdots + p_k(t)}_{\lambda \text{ times}},\]
    which yields the desired formula for $P_{\M}$. The other two are entirely analogous.
\end{proof}

Since we have proved that the coefficients of $p_k$ are always non-negative, our Theorem \ref{sparsepaving} provides a new independent proof of \cite[Conjecture 1.1]{lee} for sparse paving matroids. Moreover, we have proved the following:

\begin{cor}
    The uniform matroid $\U_{k,n}$ maximizes the Kazhdan--Lusztig coefficients, the inverse Kazhdan--Lusztig coefficients, and the $Z$-polynomial coefficients among all sparse paving matroids of rank $k$ and cardinality $n$.
\end{cor}

Observe that Theorem \ref{sparsepaving} provides a formula depending only on $n$, $k$ and $\lambda$ for each of the Kazhdan--Lusztig, the inverse Kazhdan--Lusztig and the $Z$-polynomial of all sparse paving matroids.

For instance, from \cite[Theorem 3.3]{gaoinverse} we have a formula for $Q_{\U_{k,n}}(t)$:
    \[ Q_{\U_{k,n}}(t) = \binom{n}{k} \sum_{j=0}^{\lfloor \frac{k-1}{2}\rfloor} \frac{(n-k)(k-2j)}{(n-k+j)(n-j)} \binom{k}{j} t^j\]
so that, in particular, if $\M$ is a sparse paving matroid of rank $k$ and cardinality $n$ having exactly $\lambda$ hyperplanes, then:
\begin{align}
    [t^j]Q_{\M}(t) &= \binom{n}{k} \binom{k}{j} \frac{(n-k)(k-2j)}{(n-k+j)(n-j)}\\ &\;\;- \lambda \left[ \binom{k}{j} \frac{(k+1)(k-2j)}{(1+j)(k+1-j)} - \binom{k-1}{j} \frac{k(k-1-2j)}{(1+j)(k-j)}\right].\nonumber
\end{align}

Of course an analogous (but a bit more complicated) formula holds for $P_{\M}$ using the results of \cite{gaouniform}. Also, there exist upper-bounds for $\lambda$ that are good enough to show the positivity of the expression above with little effort. 

\begin{lema}\label{numbercircuithyperplanesparsepaving}
    Let $\M$ be a sparse paving matroid of rank $k$ having $n$ elements. Then, the number of circuit-hyperplanes $\lambda$ of $\M$ satisfies:
        \[ \lambda \leq \binom{n}{k}\min \left\{\frac{1}{k+1}, \frac{1}{n-k+1}\right\}.\]
\end{lema}

\begin{proof}
   See for example \cite[Lemma 8.1]{ferroni3} or \cite[Theorem 4.8]{merino}.
\end{proof}

\begin{obs}
    There exist tighter bounds for the number of circuit-hyperplanes of a sparse paving matroid of rank $k$ and cardinality $n$ for some particular values of $k$ and $n$. In fact, this quantity coincides with the independence number of the Johnson graph $J(n,k)$, and with the maximum number of words that a binary code with word-length $n$ and constant weight $k$ can have, under the constraint of minimal distance $4$. Also, Lemma \ref{numbercircuithyperplanesparsepaving} is a weaker version of what in the coding theory literature is called the ``Johnson Bound'' (see \cite{johnson}). The exact computation of this maximum is a difficult problem, and precise values are in fact known only for few particular cases.
\end{obs}

We will write $\lambda_{k,n}$ to denote the expression $\binom{n}{k}\min \left\{\frac{1}{k+1}, \frac{1}{n-k+1}\right\}$.

\begin{teo}\label{positivitysparsepaving}
    If $\M$ is a sparse paving matroid then $P_{\M}$, $Q_{\M}$ and $Z_{\M}$ have non-negative coefficients.
\end{teo}

\begin{proof}
    For the non-negativity of $P_{\M}$ we refer to \cite[Theorem 1]{lee}, although it is possible to give an alternative proof using our formula. We included the full proof for $Q_{\M}$ in the appendix; the starting point is to split the proof into two cases, namely $2k\leq n$ and $2k>n$, to avoid working with the minimum of two fractions in every occurrence of $\lambda_{k,n}$
    
    For $Z_{\M}$ we can use the following recursion found by Proudfoot et al. in \cite[Section 4]{proudfootzeta}:
    \[ Z_{\U_{k,n}}(t) = t^k + \sum_{j=1}^k \binom{n}{k-j} t^{k-j} P_{\U_{j,n-k+j}}(t).\]
    Hence, if $\M$ has rank $k$, cardinality $n$, and exactly $\lambda$ circuit-hyperplanes, using Theorem \ref{sparsepaving} and Corollary \ref{formulitas}:
    \begin{align*} 
        Z_{\M}(t) &= Z_{\U_{k,n}}(t) - \lambda z_k(t)\\
        &= t^k + \left(\binom{n}{k-1}-\lambda(k-1)\right) t^{k-1} + \sum_{j=2}^k \left(\binom{n}{k-j}P_{\U_{j,n-k+j}}(t) - \lambda \binom{k}{j} p_j\right) t^{k-j}.
    \end{align*}
    
    We will prove that each summand in the last expression is a polynomial with non-negative coefficients. Observe that the second summand has a non-negative coefficient, since $\lambda \leq \lambda_{k,n}$ by Lemma \ref{numbercircuithyperplanesparsepaving}, and:
        \[ \lambda_{k,n} (k-1) \leq \frac{1}{n-k+1}\binom{n}{k} (k-1) = \frac{k-1}{k}\binom{n}{k-1} \leq \binom{n}{k-1}.\] 
    
    Now, if we use that $P_{\U_{j,n-k+j}} - \lambda_{j,n-k+j} p_j$ has positive coefficients (which is the first statement in this theorem), it just suffices to verify the following inequality:
        \[\binom{n}{k-j}\lambda_{j,n-k+j} \geq \lambda_{k,n} \binom{k}{j},\]
    which is just:
        \[\binom{n}{k-j}\binom{n-k+j}{j} \min\left\{\tfrac{1}{j+1}, \tfrac{1}{n-k+1}\right\} \geq \binom{n}{k}\binom{k}{j}\min\left\{\tfrac{1}{k+1},\tfrac{1}{n-k+1}\right\}.\]
    Since it is easy to verify the identity $\binom{n}{k-j}\binom{n-k+j}{j} = \binom{n}{k}\binom{k}{j}$, it suffices to show only that:
    \[\min\left\{\tfrac{1}{j+1}, \tfrac{1}{n-k+1}\right\} \geq \min\left\{\tfrac{1}{k+1},\tfrac{1}{n-k+1}\right\},\]
    which holds trivially since $j\leq k$.
\end{proof}

\section{Final remarks}\label{sec:six}

\subsection{Free bases and regularity}

Conjecture \ref{conject} asserts that connected matroids that are regular, i.e. representable over all fields, are non-degenerate. Although there is good evidence that almost all matroids are expected to possess a free basis (see the discussion in \cite[Section 7.2]{bansal}), a natural question that may arise at this point is which of these are regular and connected.

Since almost all matroids are non-representable \cite[Theorem 1.1]{nelson}, in particular almost all matroids are non-regular. However, although the family of matroids with a free basis is expected to be asymptotically predominant, the family of regular matroids with a free basis is almost negligible among the whole family of regular matroids. 

\begin{prop}\label{charoxley}
    Let $\M$ be a regular matroid with a free basis. Then $\M$ is graphic, and is obtained from a cycle graph with at least two edges by repeatedly adding a possibly empty set of parallel edges to one of the edges of the cycle.
\end{prop}

\begin{proof}
    Since $\M$ is regular, in particular $\M$ is binary. Let us call $B$ the basis of $\M$ such that $B\cup \{e\}$ is a circuit for every $e\notin B$. If $E\smallsetminus B$ consists of only one element, then the conclusion follows. Suppose then that we can pick two distinct elements $y$ and $z$ not in $B$. Since $\M$ is binary, by \cite[Theorem 9.1.2]{oxley} we have that the circuits $C_1 = B\cup\{y\}$ and $C_2=B\cup\{z\}$ are such that the symmetric difference $C_1\triangle C_2 = \{y,z\}$ is a disjoint union of circuits. Since both $B\cup\{y\}$ and $B\cup\{z\}$ are circuits, it cannot happen that either $\{y\}$ nor $\{z\}$ are circuits. The only possibility is that $\{y,z\}$ is itself a circuit. From this, it follows that the elements of $E\smallsetminus B$ are parallel to each other, and the proof is complete. 
\end{proof}

\begin{obs}
    In other words, the only matroids that are regular and contain a free basis are the matroids $\T_{k,n}$, where $1\leq k\leq n-1$, obtained by the circuit-hyperplane relaxation of $\U_{k-1,k}\oplus \U_{1,n-k}$. Observe also that in Proposition \ref{charoxley} we can change the word ``regular'' for ``binary'' and the conclusion still holds.
\end{obs}

\subsection{Modularity and non-degeneracy} 

In light of Theorem \ref{mainresult1} it is reasonable to expect now that degenerate matroids are a very restrictive class of matroids.

So far, computational experiments and partial results have yielded some examples of degenerate matroids, but up to this point they all seem to share one particular property: in some sense they are very close to being modular.

In a preliminary version of this manuscript, we left the following question.

\begin{question}[Settled by N. Proudfoot]\label{pregunta}
    Is the following assertion true?
    \[ \text{$\M$ is connected, simple and degenerate $\iff$ $\M$ is a projective geometry of rank } k\geq 3 \] 
\end{question}

It is possible to prove that the implication $\Leftarrow$ is true by noticing that projective geometries are modular, and Elias et al. proved that modular matroids are degenerate \cite[Proposition 2.14]{eliasproudfoot}. However, the implication $\Rightarrow$ is not true, as can be shown by the following example communicated to us by Nicholas Proudfoot.

\begin{ej}\label{exampleproudfoot}
    Let $\M$ be the projective geometry representable over the field $\mathbb{F}_2$ of rank $5$. Since $\M$ is modular, we know that $P_{\M}(t) = 1$, see \cite{eliasproudfoot}. According to \texttt{Sage} \cite{sagemath}, we have that $P_{\M\smallsetminus \{e\}}(t) = t + 1$ for any element $e\in E(\M)$. Also, $\M\smallsetminus \{e\}$ is a connected, simple matroid with $30$ elements and rank $5$ that is not a projective geometry (it is not modular) but is still degenerate. 
\end{ej}

\section{Appendix}\label{sec:appendix}

Here we will give the proof of Theorem \ref{positivitysparsepaving} in the remaining case of $Q_\M$. We will assume $\M$ is a sparse paving matroid of rank $k$, cardinality $n$, with $\lambda$ circuit-hyperplanes. We have already stated that:

\begin{align}\label{eqsp}
    [t^j]Q_{\M}(t) &= \binom{n}{k} \binom{k}{j} \frac{(n-k)(k-2j)}{(n-k+j)(n-j)}\\ &\;- \lambda \left[ \binom{k}{j} \frac{(k+1)(k-2j)}{(1+j)(k+1-j)} - \binom{k-1}{j} \frac{k(k-1-2j)}{(1+j)(k-j)}\right].\nonumber
\end{align}

Our claim is that when $\lambda \leq \lambda_{k,n}$, then the right-hand-side of equation \eqref{eqsp} is non-negative. In other words, for such a $\lambda$, after simplifying a factor $\binom{k}{j}$, what we assert is that:
\begin{equation}\label{hardq} \binom{n}{k}\frac{(n-k)(k-2j)}{(n-k+j)(n-j)} \geq \lambda \left[ \frac{(k+1)(k-2j)}{(1+j)(k+1-j)} - \frac{k-1-2j}{1+j}\right]. \end{equation}

We are going to prove it by separating into two cases, according to whether $2k\leq n$ or $2k > n$. The structure of the proof is as follows:
    \begin{itemize}
        \item $2k\leq n$.
            \begin{itemize}
                \item $k\geq 7$ and $j\geq 2$.
                \item $k\geq 1$ and $j = 0$.
                \item $k\geq 1$ and $j = 1$.
                \item Remaining cases (finite) by inspection.
            \end{itemize}
        \item $2k > n$.
            \begin{itemize}
                \item $k\geq 4$ and $j\geq 2$.
                \item $k\geq 1$ and $j=0$.
                \item $k\geq 1$ and $j=1$.
                \item Remaining cases (finite) by inspection.
            \end{itemize}
    \end{itemize}

Recall that we always assume that $2j \leq k$ and that $k\leq n$. The above splitting of cases helps to decide explicitly which fraction to use as $\lambda_{k,n}$, to avoid working with the minimum of two fractions.

\begin{itemize}
    \item If $2k \leq n$, then we have $\lambda_{k,n}= \frac{1}{n-k+1}\binom{n}{k}$. So, under this hypothesis, replacing $\lambda$ by its maximum possible value in equation \eqref{hardq}, it suffices to prove that:
    \begin{equation}\label{cleared0} \frac{(n-k)(k-2j)}{(n-k+j)(n-j)} \geq \frac{1}{n-k+1} \left[ \frac{(k+1)(k-2j)}{(1+j)(k+1-j)} - \frac{k-1-2j}{1+j}\right]. \end{equation}
    We can forget the minus sign and try to prove a stronger inequality instead:
    \[\frac{(n-k)(k-2j)}{(n-k+j)(n-j)} \geq \frac{1}{n-k+1}\cdot  \frac{(k+1)(k-2j)}{(1+j)(k+1-j)}. \]
    After cancelling the common factors, and clearing denominators, it suffices to show that:
    \begin{align*}
         (n-k)(1+j)(k+1-j)(n-k+1) & \geq (n-k+j)(n-j)(k+1).
    \end{align*}
    Observe that, by the inequality between the arithmetic and geometric mean we have that $(n-k+j)(n-j) \leq \left(\frac{2n-k}{2}\right)^2$. In particular, it suffices now to prove: 
    \[4(n-k)(1+j)(k+1-j)(n-k+1) \geq (2n-k)^2(k+1),\]
    which, after moving some terms, is equivalent to proving that:
    \[ \frac{4(n-k)(n-k+1)}{(2n-k)^2} (1+j)(k+1-j) \geq k+1.\]
    When one fixes the value of $n$ and considers the first fraction on the left as a function of $k$, it is easy to see (for example by differentiating) that it is strictly decreasing on the interval $[1, n/2]$, so that it assumes its minimum when $k = \lfloor\frac{n}{2}\rfloor$. In that case, its value is 
    \[\frac{4(n-n/2)(n-n/2+1)}{(2n-n/2)^2} = \frac{4}{9}+\frac{8}{9n},\] which is always greater or equal to $\frac{4}{9}$. In particular, it is enough to prove that:
    \[ \frac{4}{9} (1+j)(k+1-j) \geq k + 1,\]
    which is always true when $j\geq 2$ and $k\geq 7$. \\
    
    Now, in equation \eqref{cleared0}, we have to consider $j=0$ and $j=1$ separately.
    
    \begin{itemize}
        \item If $j=0$, equation \eqref{cleared0} is equivalent to:
            \[(n-k)k(k+1)(n-k+1) \geq (n-k)n(k+1)k-(k-1)(k+1)(n-k)n,\]
        and cancelling $(n-k)(k+1)$ in our three terms, is just:
            \[ k(n-k+1) \geq kn - (k-1) n.\]
        Since $n\geq 2k$, it suffices to prove:
        \[ k(n-k+1) \geq kn - (k-1) 2k.\]
        This allows us to cancel $k$, and get us to prove:
        \[ n-k+1 \geq n - 2(k-1).\]
        After cancelling the summand $n$, the inequality reduces to $k\geq 1$, which is obvious.
        \item If $j = 1$, equation \eqref{cleared0}, after cancelling a $(n-k+1)$ present in all terms, and clearing denominators, is just:
            \[(n-k)(k-2)2k\geq (n-1)(k+1)(k-2) - (k-3)(n-1)k.\]
        We can try to prove the stronger inequality:
            \[(n-k)(k-2)2k\geq (n-1)(k+1)(k-2) - (k-3)(n-1)(k-2),\]
        which allows us to cancel out a $k-2$ factor, and reduces to:
            \[(n-k)2k\geq (n-1)(k+1) - (k-3)(n-1).\]
        Observe that, since $n-k\geq \frac{n}{2}$, we have that $(n-k)2k \geq nk$, so that it suffices to prove:
            \[ nk + (k-3)(n-1) \geq (n-1)(k-1),\]
        which is just equivalent to:
            \[ kn+2 \geq 2n,\]
        that holds for all $k\geq 2$.
    \end{itemize}
    
    Up to this point, it is clear that we have proved that when $k\geq 7$ and $2k\leq n$, the polynomial has positive coefficients. Observe that if $k\leq 6$, then $0 \leq j \leq \lfloor\frac{k-1}{2}\rfloor \leq 2$. The only cases that are not covered by the above considerations are the ones where $j = 2$ and $k \in \{5,6\}$. Plugging $j = 2$ and $k=5$ and clearing denominators in equation \eqref{cleared0} requires to prove:
        \[ (n-5)(1)(3)(4)(n-4) \geq (n-3)(n-2)(6)(1)- 0,\]
    which is just:
        \[ 12(n-4)(n-5) \geq 6(n-3)(n-2),\]
    and this holds for all $n\geq 10$. We get an analogous inequality for $j=2$ and $k=6$. 
    
    This proves that when $2k\leq n$ equation \eqref{cleared0} and equation \eqref{hardq} are both true.
    \item If $2k > n$, then the condition on $\lambda$ is $\lambda\leq \frac{1}{k+1}\binom{n}{k}$. In particular, to prove \eqref{hardq} it suffices to show that:
        \begin{equation}\label{secondcase}
            \frac{(n-k)(k-2j)}{(n-k+j)(n-j)} \geq \frac{1}{k+1}\left[ \frac{(k+1)(k-2j)}{(1+j)(k+1-j)} - \frac{k-1-2j}{1+j}\right],
        \end{equation}
    we can use the same trick as before. First, ignore the negative term on the right-hand-side. Hence, it suffices to show that:
        \[ \frac{4(n-k)(k-2j)}{(n-k+j)(n-j)} \geq \frac{1}{k+1} \frac{(k+1)(k-2j)}{(1+j)(k+1-j)},\]
    which can be simplified to:
        \begin{equation}
            \label{facil} n-k \geq \frac{n-k+j}{1+j} \cdot\frac{n-j}{k+1-j}.
        \end{equation}
    To prove this inequality, fix $k$ and $n$ such that $2k\geq n$ and consider the right-hand-side as a function of $j$. Observe that the derivative of the right-hand-side with respect to $j$ is:
        \[ \frac{\partial}{\partial j} \text{RHS} = -\frac{(n+1)(n-k-1)(k-2j)}{(j+1)^2(k-j+1)^2},\]
    which is always negative. In particular, the right-hand-side is a decreasing function on $j$, if we fix $n$ and $k$. Assume $j\geq 2$, so that to prove \eqref{facil} it suffices to show.
        \[ n - k \geq \frac{n-k+2}{3} \cdot \frac{n-2}{k-1},\]
    This can be rewritten as:
        \[ (n-k-1)(n-3k+4)\leq 0,\]
    which holds automatically since $n\leq 2k \leq 3k-4$ for every $k\geq 4$. 
    \begin{itemize}
        \item If $j=0$, \eqref{secondcase} can be reduced to:
            \[ \frac{(n-k)k}{(n-k)n} \geq \frac{1}{k+1} \left[ \frac{(k+1)k}{(1)(k+1)} - \frac{k-1}{1}\right].\]
        This is just:
            \[ \frac{k}{n} \geq \frac{1}{k+1}, \]
        which is trivially true, since $k(k+1) = k^2+k \geq k + k \geq n$.
        
        \item If $j=1$ on \eqref{secondcase}, we have to show that:
            \[ \frac{(n-k)(k-2)}{(n-k+1)(n-1)} \geq \frac{1}{k+1}\left[ \frac{(k+1)(k-2)}{2k} - \frac{k-3}{2} \right].\]
        Observe that the the fraction $\frac{n-k}{n-k+1}$ is at least $\frac{1}{2}$. Thus, we only need to prove:
            \[\frac{k-2}{2(n-1)} \geq \frac{1}{k+1}\left[ \frac{(k+1)(k-2)}{2k} - \frac{k-3}{2} \right].\]
        After simplifying the $n$ in the denominator of the right-hand-side by using that $2k\geq n$, we reduce to:
            \[\frac{k-2}{2(2k-1)} \geq \frac{1}{k+1}\left[ \frac{(k+1)(k-2)}{2k} - \frac{k-3}{2} \right],\]
        which depends only on $k$ and can be checked to be true for $k\geq 5$.
    \end{itemize}
    All the remaining cases, which are of course finite, can be checked with a direct computation.
\end{itemize}

\section{Acknowledgements}

Both authors want to thank Prof. James Oxley for providing a nice proof for Proposition \ref{charoxley}, and Nicholas Proudfoot for making several useful suggestions and answering Question \ref{pregunta} via Example \ref{exampleproudfoot}. Also, they want to express gratitude to the reviewers of the article for their careful reading of the first version of this manuscript and several useful comments and corrections that served to enhance many important aspects of the exposition.

\bibliographystyle{amsalpha}
\bibliography{bibliography}

\end{document}